\documentclass{amsart}

\newtheorem{theorem}{Theorem}[section]

\newtheorem{corollary}[theorem]{Corollary}

\theoremstyle{definition}

\newtheorem{example}[theorem]{Example}
\theoremstyle{remark}
\newtheorem{remark}[theorem]{Remark}
\newtheorem{acknowledgment}[theorem]{Acknowledgment}

\numberwithin{equation}{section}
\begin{document}

\title{The spectrum of the restriction to an invariant subspace}
\author{Dimosthenis Drivaliaris}
\address{Department of Financial and Management Engineering\\
University of the Aegean\\
Kountouriotou 41, 82100 Chios, Greece}
\email{d.drivaliaris@fme.aegean.gr}
\author{Nikos Yannakakis}
\address{Department of Mathematics\\
National Technical University of Athens\\
Iroon Polytexneiou 9, 15780 Zografou, Greece}
\email{nyian@math.ntua.gr}



\begin{abstract}
Let $X$ be a Banach space, $A\in B(X)$ and $M$ be  an invariant subspace of $A$. We present an alternative proof that, if the spectrum of the restriction of $A$ to $M$ contains a point that is in any given hole in the spectrum of $A$, then the entire hole is in the spectrum of the restriction.
\end{abstract}

\maketitle

Let $X$ be a Banach space and $A\in B(X)$. If $M$ is a closed invariant subspace of $A$ then, as is well known, the spectrum of the restriction $A|_M$ may differ a lot from the spectrum of $A$. The following example is very characteristic.
\begin{example}
Let $X=l^2(\mathbb Z)$ and $A$ be the bilateral shift. Then the spectrum of $A$ is the unit circle. If  $M$ is the subspace of sequences  whose terms are zero for all negative integers, then the spectrum of $A|_M$ is the unit disk.
\end{example}
Hence the spectrum of the restriction may fill  possible holes in $\sigma(A)$. Recall that a hole in a compact set of the complex plane is a bounded connected component of its complement.

As was proved by J. Scroggs in \cite{scroggs} a more precise result holds: if one point of the hole lies in $\sigma(A|_M)$, then the entire hole must lie in $\sigma(A|_M)$.

This result together with the fact that the spectrum of the restriction does not intersect the unbounded connected component of $\rho(A)$ (see \cite[Corollary 4.1]{scroggs}) were obtained earlier by J. Bram in \cite{bram}, for the particular case of a normal operator in a Hilbert space. Note that this was a refinement of P. Halmos' spectral inclusion relation that  if $A$ is the minimal  normal extension of the subnormal operator $B$, then $\sigma (A)\subseteq \sigma(B)$  (see  \cite[Problem 200]{halmos}).

Our aim in this short paper is to present an alternative proof of this interesting and surprising result.
\begin{theorem}
Let $M$ be a closed invariant subspace of the bounded linear operator $A$ and $D$ be a connected component of  $\rho(A)$. If $D\cap \sigma(A|_M)\neq\emptyset$, then
\[D\subseteq \sigma (A|_M )\,.\]
\end{theorem}
\begin{proof}
Let $D$ be  a connected component of the resolvent $\rho(A)$ and  $a\in D\cap \sigma(A|_M)$.  Assume that there exists $b\in D$ with $b\in\rho(A|_M)$ and  let $C$ be any (continuous) rectifiable  path that lies in $D$  and connects $a$ and $b$.

First we show that there exists $c>0$ such that
\begin{equation}
\label{path}
\|Ax-\lambda x\|\geq c\|x\|\,,  \text{ for all }\lambda\in C \text{ and }x\in M\,.
\end{equation}
Assume the contrary;  i.e. that there exists a sequence $(\lambda_n)$ in $C$ and a sequence $(x_n)$ in $M$, with $\|x_n\|=1$, such that $\|Ax_n-\lambda_n x_n\|\rightarrow 0$, as $n\rightarrow\infty$. Then since $(\lambda_n)$ is bounded it has a subsequence, which for simplicity we denote again by $(\lambda_n)$, that converges to some $\lambda_0\in C$ (note that $C$ is the range of a continuous function and hence it is closed). But then
\[\|Ax_n-\lambda_0x_n\|\leq \|Ax_n-\lambda_nx_n\|+|\lambda_n-\lambda_0|\]
and hence $\|Ax_n-\lambda_0x_n\|\rightarrow 0$, as $n\rightarrow\infty$, which is a contradiction since $\lambda_0\in\rho(A)$.

As one may easily see inequality (\ref{path}) implies that the resolvent function 
\[R_\lambda=(A|_M-\lambda I_M)^{-1}\,,\] of the restriction $A|_M$,   is bounded on $C\cap\rho(A|_M)$ and in particular 
\[\|R_\lambda\|\leq \frac{1}{c}\,,\text{ for all  }\lambda\in C\cap\rho(A|_M)\,.\]
Hence, by the elementary properties of the resolvent function we have  
$\lambda\in\rho(A|_M)$, for all 
\[|b-\lambda|<c\leq\frac{1}{\|R_b\|}\,.\]

Since $c$ is independent of $\lambda$ the above argument shows that if two $\lambda$'s that belong to $C$ are within $c$ of each other and one is  in the resolvent set of the restriction then the other is also  in the resolvent set of the restriction. Therefore, if we divide the arc  $C$ into subarcs of length less than $c$, then the subarc containing $b$ has the other endpoint in the resolvent set of the restriction, and so on. The last subarc contains $a$ and so $a$ is also in the resolvent set of the restriction which is a contradiction. Hence
\[D\subseteq \sigma(A|_M)\,.\]
\end{proof}
\begin{remark}
(i) In his proof  J. Scroggs \cite[Theorem 4]{scroggs} uses the analyticity of the resolvent function in the connected components of $\rho(A)$ and obtains the result by using the uniqueness theorem for analytic functions  (see also the book of H. Dowson \cite[Theorem 1.29]{dowson1}). On the other hand the proof that is presented in  the book of J. Conway \cite[Theorem II.2.11 (c)]{conway} and is attributed to S. Parrott, uses the fact that the boundary of the spectrum is contained in the approximate point spectrum. 
It seems to us that our approach is simpler and in a sense more direct: it only depends on the fact that ``good'' and ``bad'' points in the same hole (and in the unbounded component) are continuously linked and thus they cannot coexist.

(ii)  P. Halmos notes  in \cite[Problem 201, p. 308]{halmos} that S. Parrott's  proof depends on the obvious  spectral inclusion
\begin{equation}
\label{halmos}
\nonumber
\sigma_{app}(A|_M)\subseteq\sigma_{app}(A)
\end{equation}
only and the conclusion holds for any pair of operators $A$ and $B$ whenever their approximate spectra are so related. Our proof may be easily adapted to this more general situation.

(iii) J. Bram  in his proof \cite[Theorem 4]{bram} for normal operators in a Hilbert space, also uses the analyticity of the resolvent function together with  the spectral theorem. An alternative proof for this case, but without using complex analysis, is given by I. Ito in \cite[Theorem 8]{ito}.

(iv) Analogous results have been obtained by H. Dowson in \cite[Theorem 1]{dowson} for operators induced on quotient spaces and by M. Putinar in \cite[Corollary 2.6]{putinar} for hyponormal operators.

(v) Our proof, without significant changes, may be adapted to prove the corresponding result for densely defined closed linear operators. A proof of this result, using the fact that the boundary of the spectrum is contained in the approximate point spectrum, was given by J. Stochel and F. Szafraniec in \cite[Theorem 2]{stochel}.
\end{remark}

An immediate corollary of the theorem  is that filling the holes in $\sigma(A)$, as in the example, is the only possibility for $\sigma(A|_M)$.
\begin{corollary}
\label{unbounded}
If $M$ is a closed invariant subspace of $A$ and $U$ is the unbounded connected component of $\rho(A)$, then
\[\sigma(A|_M)\cap U=\emptyset\,.\]
\end{corollary}
\begin{remark}
(i) This corollary  may be found in J. Scroggs' paper \cite[Corollary 4.1]{scroggs}. The proof that is presented in the book of H. Radjavi and P. Rosenthal \cite[Theorem 0.8]{radjavi} is attributed in \cite[p. 10]{radjavi} and in \cite[Lemma 2]{crimmins} to S. Parrott. As we have already mentioned an earlier proof, for normal operators, was given by J. Bram in \cite[Theorem 3]{bram}.\\
(ii) In view of the conclusion of the corollary a natural question that arises is which holes in the spectrum can be filled? A  solution to this much harder problem, for the case of normal operators in a Hilbert space, was  given by J. Conway and P. Olin  in \cite[Theorem 9.2]{conway1}.
\end{remark}
\begin{acknowledgment}
We would like to thank Prof. P. Rosenthal for his useful remarks and suggestions that have improved both the content and the presentation of this paper.
\end{acknowledgment}


\end{document}